\documentclass[12pt]{amsart}
\makeatletter
\def\@citecolor{blue}
\def\@urlcolor{blue}
\def\@linkcolor{blue}
\makeatother

\usepackage[headings]{fullpage}
\usepackage{mdwlist} 
\usepackage{lmodern}
\usepackage[T1]{fontenc}
\usepackage{amsmath,amsthm,amssymb,amsfonts,latexsym,mathrsfs}
\usepackage[all]{xy}
\usepackage{paralist}
\usepackage{color}
\usepackage{ifthen}
\usepackage[normalem]{ulem}

\makeatletter

\@addtoreset{equation}{section}
\def\theequation{\thesection.\@arabic \c@equation}
\def\@citecolor{blue}
\def\@urlcolor{blue}
\def\@linkcolor{blue}
\def\theenumi{\@roman\c@enumi}

\makeatother

\theoremstyle{plain}
\newtheorem{theorem}[equation]{Theorem}
\newtheorem*{definition*}{Definition}
\newtheorem{lemma}[equation]{Lemma}
\newtheorem{corollary}[equation]{Corollary}
\newtheorem{proposition}[equation]{Proposition}
\newtheorem*{remark*}{Remark}
\theoremstyle{definition}
\newtheorem{remark}[equation]{Remark}

\newtheorem{remarks}[equation]{Remarks}

\newtheorem{example}[equation]{Example}

\newtheorem{definition}[equation]{Definition}

\def\NZQ{\mathbb}               

\def\QQ{{\NZQ Q}}
\def\ZZ{{\NZQ Z}}
\def\RR{{\NZQ R}}
\def\frk{\mathfrak}               

\def\mm{{\frk m}}

\def\opn#1#2{\def#1{\operatorname{#2}}} 
\opn\supp{supp}
\opn\chara{char}
\opn\length{\ell}
\opn\projdim{proj\,dim}
\opn\depth{depth}
\opn\reg{reg}
\opn\lreg{lreg}
\opn\sat{^{sat}}
\opn\lex{^{lex}}
\opn\SK{Sk}
\opn\Char{char}
\opn\Ker{Ker}
\opn\Coker{Coker}
\opn\Im{Im}
\opn\Hom{Hom}
\opn\Tor{Tor}
\opn\Ext{Ext}
\opn\End{End}
\opn\Aut{Aut}
\opn\id{id}
\opn\GL{GL}
\let\:=\colon

\opn\Gin{Gin}
\opn\gin{gin}
\opn\Hilb{Hilb}
\opn\HilbPol{HilbPol}
\opn\ini{in}
\opn\End{end}

\begin{document}

\title{Zero-generic initial ideals}
\author{Giulio Caviglia}
\address{Giu\-lio Ca\-vi\-glia - Department of Mathematics -  Purdue University - 150 N. University Street, West Lafayette - 
  IN 47907-2067 - USA}
\email{gcavigli@math.purdue.edu}
\author{Enrico Sbarra}
\address{Enrico Sbarra - Dipartimento di Matematica - Universit\`a degli Studi di Pisa - Largo Bruno Pontecorvo 5 - 56127 Pisa - Italy}
\email{sbarra@dm.unipi.it}
\thanks{The work of the first author was supported by a grant from the
Simons Foundation (209661 to G. C.)}

\begin{abstract}
Given a homogeneous ideal $I$ of a polynomial ring $A={\mathbb{K}}[X_1,\dots,X_n]$ and a monomial order $\tau$, we construct a new monomial ideal of $A$ associated with $I$. We call it the zero-generic initial ideal of $I$ with respect to $\tau$  and denote it with $\gin_0(I)$, or with $\Gin_0(I)$ whenever $\tau$ is the reverse-lexicographic order. When $\chara \mathbb{K}=0$, a zero-generic initial ideal is the usual generic initial ideal.
We show that $\Gin_0(I)$ is  endowed with many interesting properties, some of which are easily seen, e.g., 
it is a strongly stable monomial ideal in any characteristic, and has the same Hilbert series as $I$; some other properties are less obvious: it shares with  $I$ the same  Castelnuovo-Mumford regularity, projective dimension and extremal Betti numbers. Quite surprisingly, $\gin_0(I)$ also satisfies  Green's Crystallization Principle, which is known to fail in positive characteristic. Thus, zero-generic initial ideals can be used as  formal analogues of  generic initial ideals computed in characteristic 0.\\
We also prove the analogue for local cohomology of Pardue's Conjecture for Betti numbers: we show that the Hilbert functions of local cohomology modules of a quotient ring of a polynomial ring modulo a weakly stable ideal are independent of the characteristic of the coefficients field. 
\end{abstract}
\date{\today}

\maketitle

\section*{Introduction and Notation}

After the founding paper of Galligo \cite{ Ga} and the results of \cite{BaSt}, generic initial ideals have become a central topic in Commutative Algebra. They are the subject of dedicated chapters in books and monographs cf. \cite{Ei}, \cite{HeHi2}, \cite{Gr} and of many research papers, for instance \cite{ArHeHi},  \cite {ChChPa}, \cite{Co}, \cite{CoSi}, \cite{CoR\"o}, \cite{Mu}, \cite{Mu2}, \cite{MuHi} and \cite{MuPu}, with topics ranging from Algebraic Geometry to Combinatorial and Computational Commutative Algebra. 
One of the main reasons why generic initial ideals have been studied so extensively in the literature after the work of Bayer and Stillman is that, when computed with respect to the reverse lexicographic order, they preserve many important invariants including the Castelnuovo-Mumford regularity. Furthermore several geometrical properties of projective varieties are encoded by  generic initial ideals, especially when computed with respect to the lexicographic order, as shown in \cite{Gr}, \cite{CoSi}, \cite{AhKwSo}, \cite{FlGr} and \cite{FlSt}.

Let $\mathbb{K}$ be any field, $I$ a homogeneous ideal of the standard graded polynomial ring $\mathbb{K}[X_1,\ldots,X_n]$, and $\tau$ a monomial order. The generic initial ideal of $I$ with respect to $\tau$ is denoted by $\gin_\tau(I).$ When $\mathbb{K}$ is infinite,  there exists a non-empty Zariski open set $U\subseteq \GL_{n}(\mathbb{K})\subseteq \mathbb{K}^{n^2}$ of coordinates changes such that $\gin_{\tau}(I)=\ini_{\tau}(gI)$ for all $g\in U$, \cite{BaSt}. In particular, $I$ and the monomial ideals  $\gin_\tau(I)$ and $\Gin(I)$ share the same Hilbert function. 
It is a consequence of a well-known upper semi-continuity argument that all graded Betti numbers and Hilbert functions of local cohomology modules do not decrease when passing from a homogeneous ideal to its initial ideal, cf. \cite{Pa1}, \cite{Sb}. Therefore, also the Castelnuovo-Mumford regularity does not decrease when taking (generic) initial ideals. 

The characteristic of the base field comes into play because generic initial ideals are Borel-fixed, and these have different combinatorial properties in characteristic zero and in positive characteristic; in the first case Borel-fixed ideals are strongly stable, when $\chara\mathbb{K}=p$ they are $p$-Borel instead \cite{Pa} (cf. also the beginning of Section \ref{local}). 

Strongly stable ideals belong to the class of stable ideals, which are well-understood, see \cite{ArHeHi}, \cite{HeHi}, \cite{HeHi2}, \cite{Se}. In fact,
a minimal graded free resolution of such ideals, called the Eliahou-Kervaire resolution, can be described easily and it is independent of the characteristic. Hence, all of the graded Betti numbers and the related invariants are also characteristic independent, which is known to be false in general for monomial ideals. In particular, the regularity of such ideals coincides with their generating degree and, by what we said above, does not depend on the characteristic. Finally, it is worth mentioning that rings which are quotients of $\mathbb{K}[X_1,\dots,X_n]$ by strongly stable ideals are sequentially Cohen-Macaulay \cite{HeSb}.
Thus, when $\chara\mathbb{K}=0$, generic initial ideals are endowed with the properties of (strongly) stable ideals; in addition, if we consider the reverse-lexicographic order, $\Gin(I)$ satisfies another fundamental property,  the so-called Crystallization Principle (cf. \cite{Gr} and the relative subsection of Section \ref{CRY}).  This principle provides a strong constrain on the degrees of a minimal set of generators of $\Gin(I)$ and allows, thus, to determine generic initial ideals in some concrete examples.

When  $\chara\mathbb{K}=p>0$, some of the properties of $\gin_\tau(I)$ do not hold true, the combinatorics underlying its structure of $p$-Borel ideal becomes more intricate, see for instance \cite{Pa}, \cite{ArHe}, \cite{EnPfPo}, \cite{HePo}, and Green's Crystallization Principle fails, see Example \ref{nocrys}. 
Therefore, the common strategy of passing to the generic initial ideal of $I$  does not work in positive characteristic that well.

Motivated by all of the above, we want to provide a tool endowed with the same  properties as a generic initial ideal computed over a field of characteristic $0$,  which might help to overcome some of the extra difficulties one can encounter in  positive characteristic. \vspace{.2cm}

Let $\mathbb{K}$ be any field and denote by $A$ the standard graded polynomial ring $\mathbb{K}[X_1,\ldots,X_n]$. Throughout the paper any assumption on $\chara\mathbb{K}$ will be specified; without loss of generality we may assume that $\mathbb{K}$ is infinite, if needed. 
When we want to stress the dependence on the coefficients field $\mathbb{K}$, we shall write $A_\mathbb{K}$ instead of $A$; accordingly, $A_{\mathbb Q}$ will denote $\mathbb Q[X_1,\dots,X_n].$ By $\mathcal{M_A}$  we shall denote the set of all (monic) monomials of $A.$ A monomial order will be denoted by $\tau$.
In order to define the zero-generic initial ideal of a given homogeneous ideal, we work with different base fields, i.e., with $\QQ$ and any field $\mathbb{K}$, and with the corresponding polynomial rings $A_\QQ$ and $A_\mathbb{K}$. Whenever $I$ is a monomial ideal of $A_\mathbb{K}$, we can assume that $I$ is generated by monic monomials, and let  $I_\mathbb{K}:=I$ and 
$I_{\mathbb{K}'}$ be the ideal generated by the image of these monomials in the ring  $A_{\mathbb{K}'}$, where $\mathbb{K}'$ is any other field.\\

Our construction of zero-generic initial ideals is elementary and it is explained in the following definition.

\begin{definition*} Let $I=I_\mathbb{K}$ be a homogeneous ideal of $A_\mathbb{K}$, and let $\tau$ be a monomial order. We define  the \emph{zero-generic initial ideal} of $I$ with respect to $\tau$ to be the ideal of $A_\mathbb{K}$  
 $$\gin_0(I):=(\gin_{\tau}((\gin_{\tau}(I))_{\QQ}))_{\mathbb{K}}.$$
We denote by $\Gin_0(I)$ the zero-generic initial ideal of $I$ with respect to the reverse lexicographic order.
\end{definition*} 

\begin{remark*}\emph{ Let $\tau$ be a monomial order on $\mathcal{M}_{A_\mathbb{K}}$; we recall the classical definition of generic initial ideal with respect to $\tau$ and some related basic facts. First, consider a matrix  of indeterminates ${\bf y}=(y_{ij})_{1\leq i,j \leq n}$ and the extension field $\mathbb{K}({\bf y})$ of $\mathbb{K}$. Let $\gamma$ be the $\mathbb{K}$-algebra homomorphism 
$\gamma: \; A_\mathbb{K} \longrightarrow A_{\mathbb{K}({\bf y})}$ defined by the assignment $X_i\mapsto \sum_{j=1}^n y_{ij} X_j$ for all $i=1,\dots,n$ and extended by linearity. Given a homogeneous ideal $I$ of $A_\mathbb{K}$, we can compute the ideal $\gamma I\subseteq A_{\mathbb{K}({\bf y})}$ and its initial ideal with respect to $\tau$, obtaining a monomial ideal $J$ of $A_{\mathbb{K}({\bf y})}$. In this way, one defines the usual generic initial $\gin_{\tau}(I)$ of $I$ to be $J_{\mathbb{K}}$.  Observe that $\mathbb{K}$ is not required to be infinite. Henceforth, when $\tau$ is the reverse-lexicographic order, $\gin_{\tau}(I)$ will be denoted  by $\Gin(I).$
}\end{remark*}

The reader accustomed to working with generic initial ideal immediately sees that $\gin_0(I)$ is invariant with respect to coordinates changes, it is Borel-fixed, it is strongly stable independently of the characteristic, it preserves the Hilbert function and if the characteristic is $0$ it coincides with $\gin(I)$. 

Encouraged by these observations, we investigate
the subject deeper by means of weakly stable ideals, a class of monomial ideals which include strongly stable, stable and $p$-Borel ideals, and, by doing so, we generalize or prove for zero-generic initial ideals some of the most significant results known on generic initial ideals. For instance, by proving in Theorem \ref{PardueCo} that the Hilbert function of local cohomology modules of weakly stable ideals are independent of the base field, we can prove  that the Hilbert functions of the local cohomology modules of $A/I$ are bounded above by those of  $A/\gin_0(I)$. As a consequence, we obtain that the projective dimension and Castelnuovo-Mumford regularity of $I$ are also bounded above by those of $\gin_0(I)$. We also show that  the Hilbert functions of the local cohomology modules of $A/\Gin(I)$ and of  $A/\Gin_0(I)$ are the same and that $I$ and $\Gin_0(I)$ have same extremal Betti numbers, as in the main result of \cite{BaChPo}, and therefore same Castelnuovo-Mumford regularity. These properties, and more, are proved in Proposition \ref{PRO}. 
In Theorem \ref{crys} we prove that Crystallization Principle holds for zero-generic initial ideals. In Proposition \ref{SCM} and Theorem \ref{composaggio} respectively, we recover criteria for an ideal to have a sequentially Cohen-Macaulay quotient ring and to be component-wise linear, providing characteristic independent analogues for zero-generic initial ideals of the main results of \cite{HeSb} and \cite{ArHeHi}. 

We also prove a lower bound for the Castelnuovo-Mumford regularity of general hyperplane sections in terms of restrictions of $\Gin_0(-)$ in Theorem \ref{restr}. In the last section we show two applications of the results we obtained: a generalization of the main results of \cite{CiLeMaRo} and a simpler proof than that of \cite{CaSb} of a well-known doubly exponential upper bound for the Castelnuovo-Mumford regularity of a homogeneous ideal in terms of its generating degree.

\section{Local cohomology and weakly stable ideals}\label{local}
In this section we  develop some technical results on Hilbert functions of local cohomology modules  with focus on a special class of monomial ideals called \emph{weakly stable}. These results are needed for the central section of this article, because they provide methods to estimate and compute several invariants of zero-generic initial ideals,  as for instance their projective dimension, Castelnuovo-Mumford regularity and extremal Betti numbers.  We shall also  discuss how extremal Betti numbers  can be computed using Local Cohomology and explain how the Hilbert functions of local cohomology modules of quotient rings defined by weakly stable monomial ideals are not affected by a change of the base field, see Theorem \ref{PardueCo}. This result provides positive answer to the analogue for Local Cohomology of a conjecture of Pardue on Betti numbers or, equivalently, on Tor modules, cf. \cite{Pa}.  Meanwhile, the original conjecture on Betti numbers has been disproved by Kummini and the first author, see \cite{CaKu}.\\

Let $n$ be a fixed integer, ${\mathbb{K}}$ an arbitrary field and $\GL_n({\mathbb{K}})$ the general linear group. We recall that the {\em Borel subgroup} of $\GL_n({\mathbb{K}})$ is the group consisting of all upper triangular matrices of $\GL_n({\mathbb{K}})$. In the following $A=A_{\mathbb{K}}$ will be the standard graded ${\mathbb{K}}$-algebra ${\mathbb{K}}[X_1,\dots,X_n]$.  An ideal of $A_{\mathbb{K}}$ is called {\em Borel-fixed} if it is fixed under the action of the Borel subgroup and, in this case, it is a monomial ideal. We also recall the definition of some other classes of monomial ideals, and to do so we need some more notation. For any $u\in\mathcal{M}$, we let $m(u)=\max\{i \: X_i| u\}$. If  $p$ is a prime
number and $k$ a non-negative integer, {\em the $p$-adic expansion of $k$} is the expression $k=\sum_i k_i p^i$, with $0\leq k_i \leq p-1$. If $\sum_i k_i p^i$ and $\sum_i l_i p^i$ are
the $p$-adic expansions of $k$ and $l$ respectively, one sets $k\leq_p l$ if
and only if $k_i\leq l_i$ for all $i$. A {\em strongly stable} (sometimes also called {\em standard Borel-fixed}) ideal $I$ is an ideal endowed with the exchange property: For each monomial $u$ of $I$, if $X_i | u$ then $X_ju/X_i\in I$  for each $j<i$. A {\em stable ideal} is defined by the property: For each monomial $u\in I$, $X_ju/X_{m(u)}\in I$ for each $j< m(u)$.  A $p$-Borel ideal is defined by the property: For each monomial $u\in I$, if $l$ is the largest integer such that $X_i^l | u$, then $X_j^ku/X_i^k\in I$, for all $j<i$ and $k \leq_p l$. 
Strongly stable ideals are Borel-fixed and if $\chara\mathbb{K} = 0$ the vice versa holds; furthermore, if $\chara {\mathbb{K}} = p$ a monomial ideal is Borel-fixed if and only if it is $p$-Borel.

We are going to recall next the definition of filter-regular sequence. A linear form $l$ of $A$ is said to be {\em filter-regular} for a graded $A$-module $M=\bigoplus_{d \in \mathbb Z} M_d$ if the multiplication map  $M_d \stackrel{\cdot l}{\rightarrow} M_{d+1}$ is  injective for all $d$ sufficiently large. A sequence of linear forms $l_1,\dots,l_r$ is called a {\em filter-regular sequence} for a graded module $M$ if $l_j$ is filter-regular for $M/(l_1,\dots,l_{j-1})M$, for all $j=1,\dots,r$. In contrast with the case of homogeneous regular sequences, the permutation of a filter-regular sequence may not be filter-regular.

In the following $\mm_A$ will denote the graded maximal ideal of $A$.

\begin{definition} A  monomial ideal $I$ of $A=A_{\mathbb{K}}$ is called {\em weakly stable} if  $X_n,\dots,X_1$ form a filter-regular sequence for $A/I$. 
\end{definition}

\noindent
Equivalently, a monomial ideal $I$ of $A$ is weakly stable if, for all monomials $u\in I$ and for all $j<m(u)$, there exists a positive integer $k$ such that $X_j^ku/X_{m(u)}^l\in I$, where $l$ is the largest integer such that $X_{m(u)}^l$ divides $u$. Furthermore, it is then easily seen that strongly stable, stable and $p$-Borel ideals are weakly stable ideals; in particular generic initial ideals are always weakly stable.

\begin{remark}\label{lasappiamolunga} {\rm\bf (a)} The definition of weakly stable ideals can be found for instance in \cite{CaSb}. This class has also been introduced by means of equivalent definitions by other authors, see that of {\em ideals of nested type} in \cite{BeGi} or  {\em quasi-stable} ideals in \cite{Se}. The name we use comes from the above exchange condition, which is weaker than  those which define stable and strongly stable ideals.\\ 
{\rm\bf (b)} Another useful characterization of weakly stable ideals is that all of their associated primes  are segments, i.e. of the form $(X_1,X_2,\ldots,X_i)$ for some $i$. We notice that $\mm_A$-primary ideals are weakly stable.\\
{\rm\bf (c)} The saturation $I:X_n^{\infty}$ of a weakly stable ideal $I$ with respect to the last variable equals the saturation $I:\mm_A^{\infty}$ of $I$ with respect to $\mm_A$ and the resulting ideal is again weakly stable.\\
{\rm\bf (d)} Let $A_{[j]}:={\mathbb{K}}[X_1,\dots,X_j]$ and $I_{[j]}$ denote the ideal $I \cap A_{[j]}$ (so that
$A_{[n]}=A$ and $I_{[n]}=I$). It descends immediately from the definition that, when $I$ is weakly stable $I_{[j]}$ is weakly stable for all $j=1,\dots,n$.
\end{remark}

Let $M$ be a finitely generated graded $A$-module, let  $\beta^A_{ij}(M):=\dim_{\mathbb{K}}\Tor_i(M,{\mathbb{K}})_j$ be the graded Betti numbers of M and $H^i_{\mm_A}(M)$ the $i^{\rm th}$ (graded)  local cohomology module of $M$ with support in the graded maximal ideal $\mm_A$ of $A$. 

In his Ph.D. Thesis \cite{Pa},  Pardue conjectured that the graded Betti numbers of $p$-Borel ideals would be independent of the characteristic of the ground field $\mathbb{K}$ or, in other words, that, for every $p$-stable ideal $I$ of $A_{\mathbb{K}}$ one would have $\beta_{ij}(I)=\beta_{ij}(I_{\mathbb Q})$ for all $i,j$. At the time, there was some evidence supporting this conjecture.  First, it is not hard to see that 
for every monomial ideal $I$ one has $\beta_{ij}(I)=\beta_{ij}(I_{\mathbb Q})$ when  $i=0,1.$
Furthermore, Pardue was able to show that important invariants that can be computed in terms of graded Betti numbers - such as the Castelnuovo-Mumford regularity and  projective dimension - of $p$-Borel ideals are characteristic independent, fact that is false in the general  case of monomial ideals.
Recently, this conjecture has been disproved in \cite{CaKu}. 

The statement of Pardue's Conjecture regards Hilbert functions of torsion modules of $p$-Borel ideals, and it makes sense to ask whether an analogous statement holds, provided that we substitute $\Tor^A_i(-,{\mathbb{K}})$ with $H^i_{\mm_A}(A/(-))$. We prove that this is indeed the case, even under the milder assumption that $I$ is  weakly stable ideal. 

\begin{theorem}[Pardue's Conjecture for local cohomology]\label{PardueCo} 
Let $I\subseteq A$ be a weakly stable ideal. Then, for all $i$,  $\Hilb\left(H^i_{\mm_A}(A/I)\right)= \Hilb\left(H^i_{\mm_{A_{\mathbb Q}}}(A_{\mathbb Q}/I_{\mathbb Q})\right),$ 
\end{theorem}

This result explains, we believe, some of the evidence that motivated Pardue's Conjecture in the first place, since the Castelnuovo-Mumford regularity, projective dimension, and extremal Betti numbers of $A/I$ are completely determined by the Hilbert functions of the local cohomology modules  of $A/I$, see Subsection \ref{extremall}.

Before proving Theorem \ref{PardueCo} we need first some technical results on weakly stable ideals and their local cohomology modules.

\begin{lemma}\label{modulosaturo} Let $I$ be a weakly stable ideal of $A_{[n]}=A$, with $n>1$. Then, 
$$I_{[n-1]}:X_{n-1}^\infty= (I:X_{n}^\infty)_{[n-1]}:X_{n-1}^{\infty}.$$
\end{lemma}
\begin{proof} First of all we notice that  $I_{[n-1]}$ and $(I:X_{n}^\infty)_{[n-1]}$ are both weakly stable; next, we recall that to prove the desired equality is equivalent to show that these two ideals agree in every sufficiently large degree $d$.  This is easily seen, since $I$ and $I:X_n^{\infty}=I:{\mm_A}^{\infty}$ agree in degree $d\gg 0$ and, therefore, their restrictions to $A_{[n-1]}$ agree as well for $d$ sufficiently large. 
\end{proof}

We recall next the following formula proved in \cite{Sb2} (see also \cite{CaSb1} equations (3.8) and (3.9)).  Let $I$ be a homogeneous ideal of $A$, $A[Z]$ be a polynomial ring over $A$ and $J$ the ideal $IA[Z]$. Then, for every  $i\geq 0$, we have 
\begin{equation}\label{AdjZ}
\Hilb\left(H^{i+1}_{\mm_{A[Z]}}(A[Z]/J)\right)=\Hilb\left(H^i_{\mm_A}(A/I)\right) \cdot \sum_{j<0} t^{j}.
\end{equation}

\noindent
The following result, which is useful for our computations, is yielded by  \eqref{AdjZ}. 

\begin{lemma}\label{scendiazero} 
 Let $I\subseteq A$ be a given weakly stable ideal. For $0<i\leq n,$
 let $J:=(I_{[n-i+1]}:X_{n-i+1}^{\infty})_{[n-i]}.$ Then, the following formula for the Hilbert function of the $i$-th local cohomology module of $A/I$ holds: 
\begin{equation}\label{azero}
\Hilb\left(H^i_{\mm_A}(A/I)\right)= \Hilb\left(H^0_{\mm_{A_{[n-i]}}}(A_{[n-i]}/J)\right) \cdot (\sum_{j<0} t^{j})^{i}.
\end{equation}
Moreover, for every $0<h\leq i\leq n$ one has
\begin{equation}\label{auno}
\Hilb\left(H^i_{\mm_A}(A/I)\right)= \Hilb\left(H^h_{\mm_{A_{[n-i+h]}}}(A_{[n-i+h]}/I_{[n-i+h]})\right) \cdot (\sum_{j<0} t^{j})^{i-h}.
\end{equation}
\end{lemma} 
\begin{proof}
Since $i>0,$ one has that $H^i_{\mm_A}(A/I)\simeq H^i_{\mm_A}(A/(I:\mm_A^{\infty}))$. Also,  
$I:\mm_A^{\infty}=I:X_n^{\infty}=(I:X_n^{\infty})_{[n-1]}A$ and, thus, \eqref{AdjZ} implies  
\begin{equation}\label{unpasso}
\Hilb(H^i_{\mm_A}(A/I))= \Hilb(H^{i-1}_{\mm_{A_{[n-1]}}}(A_{[n-1]}/(I:X_n^{\infty})_{[n-1]})) 
\cdot (\sum_{j<0} t^{j}),
\end{equation}
which is formula \eqref{azero} when $i=1.$ 
The other cases of \eqref{azero} follow by inducting on the cohomological index, considering  the ideal $(I:X_n^{\infty})_{[n-1]}$ and using Lemma \ref{modulosaturo}. 

When $i=h$, \eqref{auno} is trivial. By using \eqref{unpasso} and the same inductive argument as  before, we see that, for $0<h<i$ 
\[\Hilb(H^i_{\mm_A}(A/I))= \Hilb(H^{h}_{\mm_{A_{[n-i+h]}}}(A_{[n-i+h]}/   (I_{[n-i+h+1]}:X_{n-i+h+1}^{\infty})_{[n-i+h]}))
\cdot (\sum_{j<0} t^{j})^{i-h}.
\] 
We finally observe that $I_{[n-i+h]}$ and $(I_{[n-i+h+1]}:X_{n-i+h+1}^{\infty})_{[n-i+h]}$ have the same saturation by a repeated use of Lemma \eqref{modulosaturo}. This completes the proof since $h>0.$
\end{proof} 

\begin{proof}[Proof of Theorem \ref{PardueCo}]
If $i=0$ we have that $H^0_{\mm_A}(A/I))=(I:\mm_A^{\infty})/I$ which, by hypothesis, is just $(I:X_n^{\infty})/I$; clearly, its Hilbert function is thus independent of the base field.
When $i>0$, by Lemma \ref{scendiazero} \eqref{azero}  it is enough to compute the Hilbert function of the $0^{\rm th}$ local cohomology module of an algebra  defined by a weakly stable ideal; thus, the conclusion follows from the previous case.
\end{proof}

The definition of sequentially Cohen-Macaulay module generalizes in a sense that of Cohen-Macaulay module, cf. \cite{HeSb}. 
A necessary condition for a finitely generated graded $A$-module $M$ over a graded Gorenstein ring $A$ to be  sequentially Cohen-Macaulay is that all of its Ext-modules $\Ext^i_A(M,A)$, are either $0$ or they are Cohen-Macaulay. By \cite{HeSb}, we know that strongly stable and $p$-stable ideals are sequentially Cohen-Macaulay. We conclude this part by generalizing the statement to weakly stable ideals.

\begin{proposition}\label{seq} Let $I\subset A$ be a weakly stable ideal. Then $A/I$ is sequentially CM.
\end{proposition}
\begin{proof}
We may assume that $I\not =0$, we let $N_0=I$, $N_{i+1}=N_i:X_{n-i}^\infty$, for $i=0,\ldots,n-1$ and observe that, for all $m$, the ideal $N_{m}$ is the extension to $A$ of a non-zero ideal of $A_{[n-m]}$. By considering the sequence $0=N_0/I \subseteq N_1/I \subseteq \cdots \subseteq N_m/I$ and by removing redundant terms if they occour, we obtain a filtration that makes $A/I=N_n/I$ sequentially Cohen-Macaulay.
\end{proof}
\subsection{Extremal Betti numbers and Corners}\label{extremall} 
We recall here the definition of extremal Betti numbers  and corners of the Betti diagram. Following \cite{BaChPo}, we call a non-zero Betti number $\beta^A_{ij}(M)$ such that $\beta^A_{rs}(M)=0$ whenever $r\geq i$, $s\geq j +1$ and $s-r \geq j-i$ an {\em extremal Betti number of $M$}; moreover, we call a pair of indexes $(i,j-i)$ such that $\beta^A_{ij}(M)$ is extremal  a {\em corner of $M$}, the name being suggested by the output of the command BettiDiagram in the Computer Algebra System Macaulay2. One can see that   the extremal Betti numbers of $A/I$ can be computed directly from the local cohomology modules of $A/I$, since, by \cite{Tr} or again by \cite{BaChPo}, for any finitely generated graded $A$-module $M$
\begin{equation}\label{extremalandcohom}
\beta_{ij}^A(M)=\Hilb\left(H^{n-i}_{\mm_A}(M)\right)_{j-n}, 
\end{equation}
when $(i,j-i)$ is a corner of $M$.
\indent

We believe that the two following results are well-known to experts; they can be recovered as corollaries of Theorem \ref{PardueCo}. 

\begin{corollary}\label{extremalbase} The extremal Betti numbers of a weakly stable ideal and, thus, of $p$-Borel ideals do not depend on the base field.
\end{corollary}

For an Artinian $A$-module $M$ we let $\End(M)$ be the largest integer $j$ such that $M_j\neq 0$. With this notation, the Castelnuovo-Mumford regularity $\reg M$ of $M$ is defined as $$\reg M=\max_{i,j}\{j-i \: \beta^A_{ij}(M)\neq 0\}$$ or, equivalently, as $$\reg M =\max_i\{\End(H^i_{\mm_A}(M)) + i\}.$$

\begin{corollary}\label{regbase} The Castelnuovo-Mumford regularity and the projective dimension of a weakly stable ideal do not depend on the base field.
\end{corollary}

\section{Properties of $\rm{gin_0(-)}$ and $\rm{Gin}_0(-)$.}
This section is entirely dedicated to define and prove a list of properties of the zero-generic initial ideal, and our aim at present is to convince the reader that, in many a way, the zero-generic initial ideal can be used as a characteristic-friendly alternative to the usual generic initial ideal; in the next section we shall provide a concrete example supporting our point of view.

\begin{definition} Let $I=I_{\mathbb{K}}$ be a homogeneous ideal of $A=A_{\mathbb{K}}$, and let $\tau$ be a monomial order. We define  the \emph{zero-generic initial ideal} of $I$ with respect to $\tau$ to be the ideal of $A_{\mathbb{K}}$  
 $$\gin_0(I):=\gin_{\tau}(\gin_{\tau}(I)_{\QQ})_{{\mathbb{K}}}.$$
We denote by $\Gin_0(I)$ the zero-generic initial ideal of $I$ with respect to the reverse lexicographic order.
\end{definition} 
Recall that by \cite{Co}, one has $\gin_{\tau}(I)=\gin_\tau(\gin_{\tau}(I)),$ and thus   $\gin_0(I)=\gin_0(\gin_{\tau}(I))$ for any monomial order $\tau.$  

\begin{proposition}\label{PRO} Let $I=I_{\mathbb{K}}$ be a homogeneous ideal of $A=A_{\mathbb{K}}$. 
\begin{enumerate}
\item \label{sameh} The ideal $\gin_0(I)$ is a strongly stable ideal of $A$; $I$ and $\gin_0(I)$ have the same Hilbert function.
\item \label{char0} When the characteristic of ${\mathbb{K}}$ is $0$, $\gin_0(I)=\gin(I)$ and $\Gin_0(I)=\Gin(I)$.
\item \label{cohom} For all $i$ and $j$, the following inequality between Hilbert functions of local cohomology modules holds
$$\Hilb\left(H^i_{\mm_A}(A/I)\right)_j\leq \Hilb\left(H^i_{\mm_A}(A/\gin_0(I))\right)_j.$$
In particular, when $(i,j-i)$ is a corner of $A/\gin_0(I)$, then $$\beta_{ij}(A/I)\leq \beta_{ij}(A/\gin_0(I)).$$

\item \label{boundreg} The projective dimension and Castelnuovo-Mumford regularity of $I$ are bounded above by those of $\gin_0(I)$.
\item \label{sameloc} For all $i,$  $\Hilb\left(H^i_{\mm_A}(A/\Gin(I))\right)=  \Hilb\left(H^i_{\mm_A}(A/\Gin_0(I))\right).$

\item \label{samereg} The ideals $I$ and $\Gin_0(I)$  have the same extremal Betti numbers and, therefore, same projective dimension and Castelnuovo-Mumford regularity. 

\item \label{betti12} For $i=0,1$ and all $j$, we have $\beta_{ij}(I)\leq \beta_{ij}(\rm{gin}_0(I)).$

\end{enumerate}
\end{proposition}
\begin{proof} \eqref{sameh}: We already mentioned that in characteristic $0$ a generic initial ideal is strongly stable, and the defining exchange property is not affected when changing the field back. Also, Hilbert functions stay unchanged when taking generic initial ideals and change the coefficient field.

Part \eqref{char0} is also clear: since $\chara {\mathbb{K}}=0$, $\gin_{\tau}(I)$ is already strongly stable and, therefore, equal to $\gin_0(I)$. 

\eqref{cohom}: By \cite{Sb}, the Hilbert functions of local cohomology modules increase when taking initial ideals. The conclusion is yielded by using this fact and Theorem \ref{PardueCo} twice. The statement about extremal Betti numbers follows therefore by what we said in Subsection \ref{extremall}. 

\eqref{boundreg}: Since in a polynomial ring the projective dimension and the Castelnuovo-Mumford regularity can be read from local cohomology modules,  \eqref{boundreg} follows directly from the previous part. 

\eqref{sameloc}: Note that $\Gin(I)$ is weakly stable, and by  Proposition \ref{seq} it is sequentially CM. The desired equality follows from Theorem \ref{PardueCo} together with \eqref{HS-criterion}.

\eqref{samereg}: By \cite{BaChPo}, extremal Betti numbers are left unchanged after taking a generic initial ideal when the chosen monomial order is the revlex order. By  using \eqref{extremalandcohom} the desired equality follows directly from \eqref{sameloc}.

Finally, since the first two graded Betti numbers of a monomial ideal do not depend on the characteristic of the base field ${\mathbb{K}}$, part \eqref{betti12} is an immediate consequence of standard facts on initial ideals that we recalled in the introduction. 
\end{proof}

\begin{remark}\label{cakku}
It is reasonable to ask whether an analogue of Proposition \ref{PRO} (\ref{betti12}) is true for any homological index $i.$ Such inequality is clear only in a few special cases, for instance:  when $\Char(K)=0$, since $\gin_0(I)$ is $\gin(I)$;  when the ideal is  stable, since  a minimal free resolution of $I$ is given by the Eliahou-Kervaire resolution and therefore $\beta_{ij}(I)=\beta_{ij}(I_{\mathbb Q})$; when $\gin(I)$ is stable, by a similar reason; finally, when  $(i,j-i)$ is a corner for $\gin_0(I)$, the inequality is just a special case of  \eqref{cohom}.

In general, if we assume that there exist a homogeneous ideal $I=I_{\mathbb K}$ and indexes $i$, $j$ such that  $\beta_{ij}(I)> \beta_{ij}(\gin_0(I))$, the characteristic of $\mathbb{K}$ is necessarily $p>0$; moreover, if  we let $J=\gin(I)$, then  $\beta_{ij}(J)> \beta_{ij}(\gin_0(J))$, otherwise  $\beta_{ij}(J)\leq \beta_{ij}(\gin_0(J))=\beta_{ij}(\gin_0(I)) < \beta_{ij}(I)\leq \beta_{ij}(J)$.  Thus, if there is a counterexample to the generalization of Proposition \ref{PRO} \eqref{betti12} to any homological index $i$, this can be chosen to be a $p$-Borel ideal which is also a counterexample to the conjecture of Pardue discussed earlier.  We believe that  ideals with these properties, which can be constructed with the method found in \cite{CaKu}, could be suitable candidates. For instance, one could consider the ideal of $K[X_1,\dots,X_6]$ defined as 
\begin{align*}
J & = 
(X_1^8, X_2^{32}, X_1X_2^8X_3^{32}, X_3^{128}, X_1X_2^8X_4^{128},
X_4^{512}, X_1X_3^{32}X_5^{512}, X_2^8X_4^{128}X_5^{512},
X_3^{32}X_4^{128}X_5^{512},  \\
& \qquad 
X_5^{2048}, X_2^8x_3^{32}X_6^{2048}, X_1X_4^{128}X_6^{2048},
uX_3^{32}X_4^{128}X_6^{2048}, X_1X_5^{512}X_6^{2048},
X_2^8X_5^{512}X_6^{2048}).
\end{align*}
 On the other hand, the extremely large generating degree of such ideals make the computation of their zero-generic initial ideals challenging.
\end{remark}

\subsection{A criterion for sequentially Cohen-Macaulayness}
In \cite{HeSb}, a criterion for a quotient algebra of $A$ to be sequentially Cohen-Macaulay is given: this is the case exactly when the Hilbert functions of local cohomology modules do not change when taking the generic initial ideal with respect to the revlex order, i.e. $A/I$ is sequentially Cohen-Macaulay precisely when 
\begin{equation}\label{HS-criterion}
\Hilb\left(H^i_\mm(A/I)\right)_j=\Hilb\left(H^i_\mm(A/\Gin(I))\right)_j \hbox{\;\;\;for all\;\;\;} i,j.
\end{equation}
\noindent

The following is a straightforward consequence of the \eqref{HS-criterion} and Proposition \ref{PRO} \eqref{sameloc}.

\begin{proposition}[Criterion for sequentially Cohen-Macaulayness]\label{SCM} Let $I$ be a homogeneous ideal of $A$. Then, $A/I$ is sequentially Cohen-Macaulay if and only if the local cohomology modules of $A/I$ and $A/\Gin_0(I)$ have same Hilbert functions.
\end{proposition}

\begin{example}  For a monomial ideal, being sequentially CM may depend on the characteristic of the given base  field. We consider the Stanley-Reisner ideal $(X_1X_2X_3, X_1X_2X_5,$ $X_1X_3X_6,X_1X_4X_5, X_1X_4X_6, X_2X_3X_4, X_2X_4X_6, X_2X_5X_6, X_3X_4X_5, X_3X_5X_6)$
 in the polynomial ring  $A={\mathbb{K}}[X_1,\ldots,X_6]$ of the minimal triangulation of $\mathbb P^2_\RR$. The quotient ring $A/I$ is Cohen-Macaulay only when $\chara {\mathbb{K}}\not =2$. When ${\mathbb{K}}=\ZZ/2$ instead, $A/I$ is not sequentially Cohen-Macaulay. In fact $\Ext^3_A(A/I,A)\neq 0$ and its Krull dimension is equal to $3$, whereas $\depth \Ext^3_A(A/I,A)=2$. Thus, $\Ext^3_A(A/I,A)$ is not Cohen-Macaulay and, therefore, $A/I$ is not sequentially Cohen-Macaulay, cf. for instance \cite{HeSb} again.
 
Alternatively, recall that,  by \cite{HeReWe} and \cite{HeHi}, a square-free monomial ideal is sequentially Cohen-Macaulay if and only if its Alexander dual is component-wise linear, cf. Subsection \ref{composaggio}. It is easy to see that the Alexander dual of $I$ is $I$ itself; moreover, since $I$ is generated in degree $3$, $I$ is component-wise linear if and only if $\reg I=3$ and this is true only when $\chara{\mathbb{K}}\not=2$.
\end{example}

\subsection{The Crystallization Principle}\label{CRY} 
 One of the most useful properties of generic initial ideals, which holds true when  $\rm{char}({\mathbb{K}})=0$ and is false in general, is what Green called the {\em Crystallization Principle} in \cite{Gr}. In Theorem \ref{crys}, we prove the analogous statement  for zero-generic initial ideals without any assumption on the characteristic. We include in the following  a proof of the standard case,  cf. also \cite[29.3]{Pe}.

\noindent

\begin{theorem} [Crystallization Principle for generic initial ideals] \label{crys_0}
Let $\Char(\mathbb K)=0,$ $\tau$ a monomial order, and  $d$  an integer such that $I$ has no minimal generator of degree $d$ or larger. If $\gin(I)$ has no minimal generator in degree $d$, then it also has no minimal generator in degree larger than $d$.  
\end{theorem}
\begin{proof} Recall that $\gin(I)$ is strongly stable. By substituting $I$ with the ideal generated by its degree $d-1$ component, if necessary, we can assume that $I$ is generated in degree $d-1.$ Let $J$ be the ideal generated by $\gin(I)_{d-1}.$ We see that $J$ is strongly stable, generated in a single degree, and $\reg J={d-1}.$ Let $g$ be a change of coordinates such that $\gin(I)= \ini (gI).$ We can compute a Gr\"obner basis of $gI$ by considering the S-pairs arising from a set of minimal generators of the first syzygies module of $J$, see Algorithm 15.9 of \cite{Ei} and the discussion that precedes it.  All such syzygies are linear, and the corresponding S-pairs reduce to zero, since $(\ini gI)_{d}=J_{d}$ and  $\gin(I)$ has no generators in degree $d.$ This shows that $(gI)_{d-1}$ is spanned by a Gr\"obner basis for $gI.$ Hence $\gin(I)=J,$ as desired. 
\end{proof}

\begin{theorem} [Crystallization Principle for zero-generic initial ideals]\label{crys}
Let $\tau$ be a monomial order, and $d$ be an integer such that $I$ has no minimal generator of degree $d$ or larger. If $\gin_0(I)$ has no minimal generator in degree $d$, then it also has no minimal generator in degree larger than $d$.  
\end{theorem}

\begin{proof} 
Without loss of generality, we may  assume the base field ${\mathbb{K}}$ to be infinite, by extending it if it is needed, and that $I$ is generated in degree $d-1$.

Let $J=\ini(gI)=\gin(I)$ for  a general linear change of coordinates $g$,  and assume  that $\gin_0(I)=\gin(J_\QQ)_{\mathbb{K}}$ has no minimal generator in degree $d$. Notice that $\gin(J_\QQ)$ and, thus, $J_\QQ$ and $J$  have no minimal generator of degree $d$ as well. 

We now denote by $P=P_\QQ$ the generic initial ideal of 
$(J_{d-1})_{\mathbb Q}$, and we observe that, by assumption, $P$ agrees with $\gin(J_{\mathbb Q})$ in degree $d$; thus, $P$ has no minimal generator in degree $d$  and,  by Theorem \ref{crys_0}, $P$ is generated in degree $d-1$. Furthermore,  $P$ is strongly stable and, thus, its Castelnuovo-Mumford regularity is precisely $d-1$. Moreover, $\beta_{1,j}(P)\geq \beta_{1,j}((J_{d-1})_\mathbb{Q})= \beta_{1,j}(J_{d-1})$ since the first Betti numbers of monomial ideals do not depend on the characteristic.
 Hence, the first syzygies of  $(J_{d-1})$ are linear. By arguing as in the proof of the previous theorem, it follows that $J$ is generated in degree $d-1.$ Thus $\gin(J_{\mathbb Q})=P$ and
this yields that $\gin_0(I)=(\gin(J_\QQ))_{\mathbb{K}}=P_{\mathbb{K}}$ is generated in degree $d-1$, as desired. 
\end{proof}

The reason why the above theorem is unexpected is due to the fact that, by definition, computing $\gin_0(I)$ when the characteristic of the field if positive requires as an intermediate step the calculation of $\gin(I)$, which does not satisfy  Crystallization (with respect to $I$), as we show in the following example. 
  
\begin{example}\label{nocrys} 
Let ${\mathbb{K}}$ be a field of characteristic $3$, and let $J$ be the ideal of ${\mathbb{K}}[X_1,X_2]$  generated by $(X_1^6,X_2^6).$ One can verify, as we will explain below, that for a monomial order such that $X_1>X_2$ one has $\gin(J)=(X_1^6,X_1^3X_2^3,X_2^9).$ Notice a gap in degree $7$ and $8,$ where there is no minimal generator.
On the other hand, since $\gin_0(J)$ is strongly stable, it is equal to the lex-segment ideal with the same Hilbert function as $J$, that is  $(X_1^6,X_1^5X_2,X_1^4X_2^3,X_1^3X_2^5,X_1^2X_2^7,X_1X_2^9,X_2^{11}).$ 

By using the Frobenius map it is easy to create many examples of ideals for which $\gin(I)$ does not satisfy the Crystallization Principle. Let ${\mathbb{K}}$ be a field of characteristic $p$ and let $F$ be the Frobenius  map $A \rightarrow A$ such that $f\mapsto f^p.$ For any monomial order $\tau$, the map $F$ commutes with computing generic initial ideal (see \cite{CaSb}), i.e., $F(\gin_{\tau}(I))=\gin_{\tau}(F(I))$. For instance, in the example above $J$ is $F(X_1^2,X_2^2)$, and  $\gin(X_1^2,X_2^2)$ is the lex-segment ideal $(X_1^2,X_1X_2,X_2^3).$

For every homogeneous ideal $I$ with generators in degree at most $d$ and such that $\gin_{\tau}(I)$ has a generator in degree $d+1$, we see that $F(I)$ is generated in degree at most $pd$, and $\gin_{\tau}(F(I))$ has a generator in degree $pd+p$ and no generator in degree $pd+p-1$. 
\end{example}

\subsection{A criterion for component-wise linearity}\label{composaggio}
A consequence of Theorem \ref{crys}, and of the method used in its proof, is a characteristic-free adaptation of a result which we believe to be well-known in characteristic zero .

We recall that a homogeneous ideal $I$ of $A$ is \emph{component-wise linear} if, for every degree $d,$ the  ideal $I_{\vert d}$ generated by $I_d$, has a linear graded free resolution. This is equivalent to saying that for every $d$, $I_{\vert d}$ is either $0$ or has Castelnuovo-Mumford regularity equal to $d.$ \\
Component-wise linear ideals have many remarkable properties, see \cite{ArHeHi}, \cite{HeHi}, \cite{Co} and \cite{NaR\"o}. 
In particular, \cite{HeReWe}, \cite{HeHi} proved that a squarefree monomial ideal is component-wise linear if and only if its Alexander dual is sequentially CM. Also, it was proved in \cite{ArHeHi} that,  when $\Char(\mathbb K)=0,$ an ideal $I$ is component-wise linear if and only if  $I$ and $\Gin_0(I)$ have the same number of minimal generators. By using $\Gin_0(I)$, we present below the analogue, in any characteristic, of this result.

\begin{theorem}[Criterion for component-wise linearity]\label{CompLin}
An ideal $I$ is component-wise linear if and only if $I$ and $\Gin_0(I)$ have the same number of minimal  generators. Moreover, when $I$ and $\gin_0(I)$ have same number of generators, $I$ is component-wise linear.
\end{theorem}
\begin{proof} 
For all $d$,  $\Gin_0(I_{\vert d})$ has no minimal generator in degree $d+1$ by Proposition \ref{PRO} (\ref{samereg}) and,  therefore,  $(I_{\vert d})_{d+1}$ and $(\rm{Gin}_0(I_{\vert d}))_{d+1}$ have the same dimension; thus, the Hilbert functions of $I/\mm_AI$ and $\Gin_0(I)/\mm_A\rm{Gin}_0(I)$ are the same and, consequently, $I$ and $\Gin_{0}(I)$ have same number of minimal generators.

Now, let $\tau$ be a monomial order and assume that $I$ and $\gin_0(I)$ have same number of minimal generators.
Proposition \ref{PRO} (\ref{betti12}) yields $\beta_{0j}(I)=\beta_{0j}(\gin_0(I))$, for all $j$ and this,  together with Proposition \ref{PRO} \eqref{sameh}, implies that the Hilbert functions of $\mm_AI$ and $\mm_A\gin_0(I)$ are the same. Equivalently, for every  $d$,   $\gin_0(I_{\vert d})$ has no minimal generator in degree $d+1.$ Theorem (\ref{crys}) implies that when it is not zero,  $\gin_0(I_{\vert d})$  is generated in degree $d$; hence,  $\gin_0(I_{\vert d})$ has  regularity $d$ since it is strongly stable. By Proposition \ref{PRO} (\ref{boundreg}), the ideal $I_{\vert d}$ is either zero or it has regularity $d$, as desired.
 \end{proof}

\subsection{Castelnuovo-Mumford regularity of general restrictions}
One important property of the reverse lexicographic order is that taking initial ideals commutes with respect to going modulo the last variables, as we recalled in the introduction.  As a consequence, generic initial ideals with respect to such an order give some  information also on restrictions to general linear spaces. Throughout this section, thus, we let $\tau$  be the reverse lexicographic order. It is not restrictive to assume, and we do, that  $\vert {\mathbb{K}} \vert =\infty.$ \vspace{.1cm}

Let $l_n,\dots,l_{i+1}$ be linear forms of $A$ such that $X_1,\dots,X_i,l_{i+1},\dots,l_n$ form an ordered basis of $A_1$.  By defining a change of coordinates $g$ which maps this basis to $X_1,\dots,X_n$, given  any homogeneous ideal $I$ of $A$, we let the  \emph{ restriction of $I$ to $A_{[i]}$ with respect to $l_n,\dots,l_{i+1}$} to be the image of $gI$ in $A_{[i]}$ via the isomorphism $A/(X_n,\dots,X_{i+1})\simeq {\mathbb{K}}[X_1,\dots,X_i]$.

\begin{definition} We say that  \emph{a general restriction of $I$ to $A_{[i]}$  satisfies a property $\mathcal{P}$} if there exists a non-empty Zariski open set of $(\mathbb {P}^{n-1})^{n-i}$ whose points $([l_n],...,[l_{i+1}])$ are such that the restriction of $I$ to $A_{[i]}$ with respect to  $l_n,\dots,l_{i+1}$ satisfies $\mathcal{P}$.
\end{definition}\label{needed}

\begin{remark} It is relevant for the following to notice that, for a homogeneous ideal $I$ of $A$ and a general restriction  $J$ of $I$ to $A_{[i]}$, the ideal
$\Gin(J)$ is well-defined. In fact,  the ideal  $\Gin(J)$ is equal to $\Gin(I)_{[i]}$, which is  the general restriction of $\Gin(I)$ to $A_{[i]}$, see for instance \cite{Gr} Theorem 2.30 (4). Hence, 
\begin{equation} \label{GIN-REG}
\reg A/(I+ (l_n,\dots,l_{i+1}))=\reg A/(\Gin(I)+(X_n,\dots, X_{i+1}))= \reg A_{[i]}/ \Gin(I)_{[i]},
\end{equation}
for general linear forms $l_n,\dots,l_{i+1}$; moreover, $\reg J=\reg \Gin(I)_{[i]}$.
\end{remark}

From the above observations, we can conclude that,  for a homogeneous ideal $I$ of $A$ and a general restriction  $J$ of $I$ to $A_{[i]}$, also $\Gin_0(J)$ is well-defined; unfortunately, though, $\Gin_0(J)$ is not the general restriction to $A_{[i]}$ of $\Gin_0(I)$. We observe that the latter is the ideal  $\Gin_0(I)_{[i]}$, since $\Gin_0(I)$ is strongly stable. Therefore the second equality in \eqref{GIN-REG} is still valid for $\Gin_0(\cdot)$, whereas the first one  is false in general. The following example illustrates such a situation. 

\begin{example}
Let $I=(X_1^2,X_2^2,X_3^2)\subset A={\mathbb{K}}[X_1,X_2,X_3]$ and $\chara{\mathbb{K}}=2$. Since the ideal $I$ is $2$-Borel, $\Gin(I)=I$ and also the general restriction $J$ of $I$ to $A_{[2]}$ is $(X_1^2,X_2^2).$ Moreover, $\Gin_0(J)=(X_1^2,X_1X_2,X_2^3)$ whereas $\Gin_0(I)_{[2]}= (X_1^2,X_1X_2,X_2^2)$. Furthermore,  $2=\reg A/(I+(l_3))> \reg A/(\Gin_0(I)+(X_3))=1$. 
\end{example}

We are going to show in Theorem \ref{restr} that one inequality is still valid and it provides a lower bound for the regularity of general restrictions in terms of zero-generic initial ideals. To this purpose, we prove first  a technical fact which will be crucial in our proof.  

\begin{proposition}\label{miracle}  Let $I$ be a weakly stable ideal of $A$. Then, for all $j=1,\dots,n$,  
\begin{eqnarray} 
\label{EQ2} \Hilb(H_{\mm_{A_{[j]}}}^0(A_{[j]}/I_{[j]}))  \geq  \Hilb(H_{\mm_{A_{[j]}}}^0(A_{[j]}/\Gin(I)_{[j]})), \text{ and }\\
\label{EQ1}
\Hilb(H_{\mm_{A_{[j]}}}^i(A_{[j]}/I_{[j]}))=  \Hilb(H_{\mm_{A_{[j]}}}^i(A_{[j]}/\Gin(I)_{[j]})) \text{ for all }i>0. 
\end{eqnarray}
\end{proposition}
\begin{proof} We first prove \eqref{EQ1}. If $i>j$ there is nothing to prove.
Assume $0< i\leq j$ and observe that, by  Lemma \ref{scendiazero} \eqref{auno}, it is enough to show that $\Hilb\left(H_{\mm_{A}}^{n+i-j}(A/I)\right)=  \Hilb\left(H_{\mm_{A}}^{n+i-j}(A/(\Gin(I)))\right)$; by Proposition \ref{seq}, $I$ is sequentially Cohen-Macaulay and, thus, this is achieved by applying Herzog-Sbarra's Criterion \eqref{HS-criterion}.

Next, we show \eqref{EQ2} and to do so we first recall the following formula due to Serre (see for instance \cite{BrHe} Theorem 4.4.3). 
Let $UP$ be a homogeneous ideal of a standard graded ${\mathbb{K}}$-algebra $R$. Then or every degree $d$,
\begin{equation}
\label{Serre}
\dim_K(R/U)_d -  \HilbPol_{R/U}(d)=\sum_{i \geq 0} (-1)^i \dim_K H_{\mm_{R}}^{i}(R/U)_d, 
\end{equation}
where $\HilbPol_{R/U}$ denotes the Hilbert polynomial of $R/U$. Now, notice that $\Hilb(A_{[j]}/I_{[j]})=\Hilb(A/(I+(X_n,\dots,X_{j+1}))$ and $\Hilb(A_{[j]}/\Gin(I)_{[j]})=\Hilb(A/(\Gin(I)+(X_n,\dots,X_{j+1}))$, since $I$ and $\Gin(I)$ are monomial ideals.
Furthermore, $\Hilb(A/(\Gin(I)+(X_n,\dots,X_{j+1}))=\Hilb(A/(\ini(gI)+(X_n,\dots,X_{j+1}))$ where $g$ is a general linear change of coordinates. By a well-known property of the reverse lexicographic order
the latter is equal to $\Hilb(A/(\ini(gI+ (X_n,\dots,X_{j+1}))).$  Thus, 
$\Hilb(A_{[j]}/\Gin(I)_{[j]})=\Hilb(A/(gI+ (X_n,\dots,X_{j+1})))=\Hilb(A/(I+(l_n,\dots,l_{j+1}))$ for $l_n,\dots,l_{j+1}$ general linear forms.  In particular, since the $l_n,\dots,l_{j+1}$ are general, we have
$\Hilb(A/(I+(X_n,\dots,X_{j+1})) \geq \Hilb(A/(I+(l_n,\dots,l_{j+1}))$, which now yields
$$\Hilb(A_{[j]}/I_{[j]})\geq \Hilb(A_{[j]}/\Gin(I)_{[j]}).$$
Hence, by \eqref{EQ1} and \eqref{Serre}, we are left to prove that $A_{[j]}/I_{[j]}$ and  $A_{[j]}/\Gin(I)_{[j]}$ have the same Hilbert polynomial or, equivalently, it is enough to verify that for all $d$ sufficiently large, $(A_{[j]}/I_{[j]})_d$ and $(A_{[j]}/\Gin(I)_{[j]})_d$ have the same dimension. To this purpose, we just need to observe that $A/I$ and $A/\Gin(I)$ have the same Hilbert function  and  that $X_n,\dots,X_{j+1}$ is a filter-regular sequence on both rings, since $I$ and $\Gin(I)$ are weakly stable.
 \end{proof}

From the definition of regularity via local cohomology modules, we derive immediately the following corollary.
\begin{corollary}\label{regmiracle}
Let $I$ be a weakly stable ideal of $A$. Then, for all $j=1,\dots,n$,
$$\reg A_{[j]}/I_{[j]} \geq  \reg A_{[j]}/\Gin(I)_{[j]}.$$ 
\end{corollary}

We are now ready to prove the following theorem. 

\begin{theorem}[Castelnuovo-Mumford regularity and general restrictions] \label{restr} Let $I$ be a homogeneous ideal of $A$ and let $l_n,\dots,l_{i+1}$, $0\leq i< n$, general linear forms. Then, 
$$\reg A/(I+(l_n,\dots,l_{i+1})) \geq
\reg A/(\Gin_0(I)+(X_n,\dots,X_{i+1}))= \reg A_{[i]}/\Gin_0(I)_{[i]}.$$
\end{theorem}
\begin{proof} We already motivated the validity of the second equality and are left to prove the above inequality. First, by \eqref{GIN-REG}, 
$\reg A/(I+(l_n,\dots,l_{i+1}))=\reg  A/(\Gin(I)+(X_n,\dots,X_{i+1}))$ and, since $\Gin(I)$ is weakly stable,  Corollary \ref{regbase} implies that the left-hand side of the inequality  is equal to 
$\reg  A_\QQ/(\Gin(I)_\QQ+(X_n,\dots,X_{i+1}))$. On the other hand, $\reg A/(\Gin_0(I)+(X_n,\dots,X_{i+1}))=
\reg A/(\Gin(\Gin(I)_\QQ)_{\mathbb{K}}+(X_n,\dots,X_{i+1}))$ and, again by Corollary \ref{regbase}, equal to
$\reg A_\QQ/(\Gin(\Gin(I)_\QQ)+(X_n,\dots,X_{i+1}))$. The conclusion is now a straightforward consequence of Corollary
\ref{regmiracle}, applied to the ideal $\Gin(I)_\QQ$ in the ring $A_\QQ$.
\end{proof}
\section{Bounds for the Castelnuovo-Mumford regularity}
In this final section, we provide two applications of the results we proved so far.

\subsection{Lower bounds}
Recently, in \cite{CiLeMaRo}, lower bounds for the Castelnuovo-Mumford regularity of saturated ideals with fixed Hilbert polynomial have been proven in characteristic zero, but the assumption on the characteristic can now be dropped: Part \eqref{sameh} and \eqref{samereg} of Proposition \ref{PRO}, and Corollary \ref{extremalbase} yield the following remark, which, in turn, implies the next theorem.

\begin{remark} Let $\mathbb{K}$ and $\mathbb{F}$ be any two fields. Given a homogeneous ideal $I\subset A_{\mathbb K}=\mathbb{K}[X_1,\dots,X_n]$ there exists a strongly stable ideal of $J\subset A_{\mathbb F}=\mathbb{F}[X_1,\dots,X_n]$ such that $I$ and $J$ have same Hilbert function, extremal Betti numbers and, therefore, same projective dimension and Castelnuovo-Mumford regularity. Since $I$ and $J$ have the same projective dimension, $I$ is saturated if and only if $J$ is saturated.
\end{remark}

\begin{theorem}\label{cioffi}
\cite[Theorem A]{CiLeMaRo} holds in any characteristic.
\end{theorem}

\subsection{Upper bounds} We provide, as an application of $\Gin_0(-)$,  a  new characteristic-free proof of a well-known doubly exponential bounds for the Castelnuovo-Mumford regularity of an ideal in terms of its generating degree. By Theorem \ref{restr}, our line of reasoning follows now closely that of Galligo's original proof for the characteristic zero case \cite{Ga}.\\

Let $I$ be a non-zero homogeneous ideal of $A$. We denote by $D(I)$ the generating degree of $I$, i.e. the maximum degree of a minimal generator of $I$; we also let $\mu(I)$ be the number of minimal generators of $I$. In particular,  
$$\mu(I)=\sum_{j}\beta_{0j}(I) \text{\;\;\; and\;\;\; } D(I):= \max \{j \: \beta_{0j}(I)\not = 0 \}\leq \reg I.$$

The following lemma is a straightforward consequence of the combinatorial properties of strongly stable ideals, 
see again \cite{CaSb}, Proposition 1.6, for a generalization to weakly stable ideals.
\begin{lemma} \label{PROD} Let $I$ be a strongly stable ideal of $A={\mathbb{K}}[X_1,\dots,X_n]$. Then
\[\mu(I)\leq \prod_{i=1}^{n-1} (D(I_{[i]})+1).\]
\end{lemma}

In the proof of the following theorem we shall denote by $U_{\langle j \rangle}$, for $1\leq i\leq n$ and $1\leq j\leq i$ a general restriction to $A_{[j]}$ of a homogeneous ideal $U$ of $A_{[i]}$, so that $U_{\langle n \rangle}=U$. We notice that, by Remark \ref{needed}, $\Gin_0(U_{\langle j \rangle})$ is well-defined.

\begin{theorem} Let $I$ be a homogeneous ideal of $A={\mathbb{K}}[X_1,\dots,X_n],$ with $n \geq 2$. Then
\[\reg I\leq (2D(I))^{2^{n-2}}.\]
\end{theorem}
\begin{proof} The statement is trivial if $D(I)\leq 1$ and, thus, we may assume $D(I)\geq 2$. Let $J$ be $\Gin_0(I)$, and recall that by Proposition \ref{PRO} \eqref{samereg}, $\reg I=\reg J.$ Since $J$ is strongly stable, its regularity is equal to $D(J)$ and, for the same reason,  $\reg J_{[i]} =D(J_{[i]})$ for all $1\leq i \leq n$.

 Let now $J_{(i)}$ denote the ideal  $\Gin_0(I_{\langle i \rangle})$,  for all $1\leq i <n$. 
By Theorem \ref{restr}, $D(J_{[i]})$ is bounded above by $\reg I_{\langle i \rangle}$ and, for $i\leq j$,  $D((J_{(j)})_{[i]})\leq \reg (I_{\langle j \rangle})_{\langle i \rangle}= \reg I_{\langle i \rangle}$. 
Together with  Lemma \ref{PROD}, this implies 
$$ \mu (J_{(j)}) \leq  \prod_{i=1}^{j-1} (D((J_{(j)})_{[i]})+1)\leq  \prod_{i=1}^{j-1} (\reg I_{\langle i \rangle}+1).$$
By Proposition \ref{PRO} \eqref{samereg},  $$\reg I_{\langle j \rangle} = \reg \Gin_0(I_{\langle j\rangle}) = D(J_{(j)}),$$ 
and, furthermore,
\begin{equation}\label{LAST?}
 D(J_{(j)})
\leq D(I_{\langle j \rangle})+\mu(J_{(j)})-1
\leq D(I)+\mu(J_{(j)})-1 
\leq D(I)-1 + \prod_{i=1}^{j-1} (\reg I_{\langle i \rangle}+1),
\end{equation}
where the first inequality is a straightforward application of the Crystallization Principle, see Theorem \ref{crys}.

As in the proof of \cite{CaSb}, Corollary 1.8, we set $B_1=D(I)$, and recursively $B_j=D(I)-1+ \prod_{i=1}^{j-1}(B_i+1)$, for all $1<j\leq n$. It is easy to
see that $B_j\leq B_{j-1}^2$  and,  thus,  $B_j\leq (B_2)^{2^{j-2}}$, for all $j\geq 2$. 
An easy induction together with \eqref{LAST?} implies that $\reg I_{\langle j \rangle} \leq B_j$, for all $1\leq j\leq n$. 
Hence, $\reg I=\reg I_{\langle n \rangle}$ is  bounded above by $(B_2)^{2^{n-2}}=(2D(I))^{2^{n-2}}$, as desired.
\end{proof}


\end{document}